\documentclass[a4paper,11pt]{amsart}

\usepackage{amsthm}
\usepackage{amssymb}
\usepackage{latexsym}
\usepackage{euscript}
\usepackage{amsmath}
\usepackage{amstext}
\usepackage{amsgen}
\usepackage{amsbsy}
\usepackage{amsopn}
\usepackage{amsfonts}
\usepackage{hyperref}
\usepackage{graphicx}
\usepackage{bbm}
\usepackage{tikz}
\usepackage{tikz-cd}
\usepackage{enumerate}
\usepackage{color}
\usepackage{a4wide}
\usepackage{mathrsfs}

\usetikzlibrary{matrix,arrows,decorations.pathmorphing}

\newtheorem{theorem}{Theorem}[section]
\newtheorem{proposition}[theorem]{Proposition}
\newtheorem{corollary}[theorem]{Corollary}
\newtheorem{lemma}[theorem]{Lemma}

\newtheorem*{mainthm}{Main Theorem}

\theoremstyle{definition}

\newtheorem{remark}[theorem]{Remark}

\newcommand{\defn}[1]{\emph{#1}}

\newcommand{\ie}{i.e.\ }
\newcommand{\eg}{e.g.\ }

\newcommand{\Z}{\mathbb{Z}}

\newcommand{\R}{\mathbb{R}}
\newcommand{\C}{\mathbb{C}}

\newcommand{\PP}{\mathbb{P}}

\newcommand{\im}{\mathrm{im}}

\newcommand{\cat}[1]{\mathrm{#1}}
\renewcommand{\mod}[1]{{#1}\mathrm{-mod}}

\newcommand{\loc}[1]{\cat{Loc}\!\left( #1\right)}
\newcommand{\perv}[2]{{}^{#1}\cat{Perv}\!\left(#2\right)}
\newcommand{\projperv}[2]{{}^{#1}\cat{Proj}\!\left(#2\right)}
\newcommand{\pfun}[2]{{}^{#1}{#2}}

\newcommand{\ic}[2]{\pfun{#1}{\sh{IC}_{#2}}}

\newcommand{\constr}[2]{\cat{D}_{#1}({#2})}

\newcommand{\id}{id}

\newcommand{\mor}[3]{{\mathrm{Hom}^{#1}}\!\left(#2,#3\right)}
\newcommand{\Mor}[3]{{\mathrm{Hom}}_{#1}\!\left(#2,#3\right)}
\newcommand{\ext}[3]{{\mathrm{Ext}^{#1}}\!\left(#2,#3\right)}
\newcommand{\Ext}[4]{{\mathrm{Ext}_{#1}^{#2}}\!\left(#3,#4\right)}
\newcommand{\End}[1]{\mathrm{End}\!\left(#1\right)}

\newcommand{\sh}[1]{\mathcal{#1}}

\newcommand{\dual}{\mathcal{D}}

\newcommand{\pt}{{\bf pt}}

\newcommand{\DisTriang}[6]{
\begin{tikzcd}[row sep = small, column sep =small]
\end{tikzcd}
}
\newcommand{\epic}{\twoheadrightarrow}

\newcommand{\proj}[2]{\pfun{#1}{\sh{P}_{#2}}}

\title{Perverse sheaves and finite-dimensional algebras}
\author{Alessio Cipriani}
\author{Jon Woolf}

\address{Alessio Cipriani, Dept. of Mathematical Sciences, University of Liverpool, L69 7ZL, U.K.}
\email{A.Cipriani@liverpool.ac.uk}
\address{Jon Woolf, Dept. of Mathematical Sciences, University of Liverpool, L69 7ZL, U.K.}
\email{Jonathan.Woolf@liverpool.ac.uk}

\begin{document}

\maketitle
\section{Introduction}

Let $X$ be a topologically stratified space. The category $\perv{p}{X}$ of $p$-perverse sheaves on $X$, constructible with respect to the given stratification and with coefficients in a field $k$, captures interesting aspects of the topology of $X$ and its stratification. When $X$ is a complex algebraic variety with an algebraic stratification and $p(S) = -\dim_\C(S)$ is the middle-perversity, the perverse sheaves have deep connections with Morse theory, differential equations ($\mathscr{D}$-modules) and, for suitable $X$, with representation theory. 

Perverse sheaves are defined as the heart of a t--structure on the constructible derived category $\constr{c}{X}$ cut out by imposing cohomological vanishing conditions. Whilst this is convenient for theoretical purposes, one often wants a more explicit description, as modules over an algebra or as quiver representations. We characterise those $X$ for which the perverse sheaves can be described as modules over a finite-dimensional algebra. 
\begin{mainthm}[Corollary \ref{main result}]
The category $\perv{p}{X}$ is equivalent to the category of finite-dimensional (left) modules over a finite-dimensional $k$-algebra if and only if $X$ has finitely many strata and the same holds for the category of local systems on each of these strata.
\end{mainthm}
Surprisingly this result depends only upon the fundamental groups of the strata and not on the perversity or the way the strata are assembled. In particular, if $X$ has finitely many strata each with finite fundamental group then, for any perversity $p$ and field $k$, perverse sheaves can be described as finite-dimensional modules over a finite-dimensional algebra.

The key component of the proof is Theorem \ref{sufficient} in which we construct projective covers of simple perverse sheaves. Our approach generalises that in \cite[\S 3.2]{BGS} and is motivated by reverse-engineering. Effectively we assume that $\perv{p}{X}$ is equivalent to modules over a finite-dimensional algebra, specifically that perverse sheaves have a quiver description, and attempt to construct the projective cover in the same way one would for a quiver representation. This turns out to be possible when $X$ has finitely many strata and the category of local systems on each has enough projectives. The other ingredients of the proof are standard results about projective covers and generators, and elementary observations about the behaviour of projective perverse sheaves under recollement functors.

In principle one can use the construction of Theorem \ref{sufficient} to find a projective generator $\sh{P}$ for $\perv{p}{X}$, but in practice this is difficult because it is not obvious how to implement the construction as an effective algorithm. We discuss approaches to computing a projective generator in the final section, but leave the details of implementation for a second paper.

There is an extensive literature on algebraic and quiver descriptions of perverse sheaves and to provide some context we survey  some of the main themes. Unless otherwise stated, the results below hold for the middle-perversity $p(S) = -\dim_\R(S)/2$.  Be\u{\i}linson \cite{MR923134}, see also \cite{MR2671769}, uses the nearby and vanishing cycle functors to describe how perverse sheaves on a variety are glued from perverse sheaves on a hypersurface and its complement. This glueing construction can be used to obtain quiver descriptions for perverse sheaves on an algebraic curve, and also in various higher-dimensional cases. 

MacPherson and Vilonen \cite{MR833195} give a similar glueing construction for perverse sheaves on the complement of a closed stratum of a Whitney stratified space. This, together with micro-local techniques and deformation to the normal cone, is the key ingredient in \cite{MR1390658} in which they, together with Gel'fand, prove that the category of perverse sheaves on a stratified analytic variety is equivalent to the module category over a finitely-presented algebra. These ideas, together with micro-local Morse theory are used by Braden \cite{MR1900761}, and Braden and Grinberg \cite{MR1666554}, to obtain quiver descriptions for perverse sheaves on Schubert-stratified Grassmannians, and, respectively on rank stratifications of matrices. 
Prior to this, Be\u{\i}linson, Ginzburg and Soergel \cite{BGS} had used algebraic techniques to show that perverse sheaves on these spaces could be described as modules over a finite-dimensional Koszul algebra, and hence have a description as representations of a quiver with quadratic relations.

Another case in which Koszul algebra plays a prominent role is that of perverse sheaves on a triangulated space, now for any `classical' perversity satisfying $\dim_\R(S) - \dim_\R(T) \leq p(T)-p(S)\leq 0$. By utilising the extra combinatorial structure Polishchuk \cite{MR1453053} shows that these are representations of a Koszul algebra, and Vybornov \cite{MR1666864} uses this and the Koszul duality results of \cite{BGS} to describe the perverse sheaves as the constructible sheaves  with respect to a stratification by `perverse simplices'. Using similar techniques, but now back in the algebro-geometric setting and for the middle-perversity, Vybornov \cite{MR2264803} obtains a quiver description for perverse sheaves on Schubert-stratified flag varieties.  Continuing in a combinatorial vein, Kapranov and Schechtmann \cite{MR3450484} obtain a quiver description, with monomial relations, for perverse sheaves on complexified hyperplane arrangements. (Even in the simplest case of $\C$ stratified by $0$ and $\C^*$ their description differs from Be\u{\i}linson's glueing description.) Their construction is motivated by the earlier work of Galligo, Granger and Maisonobe on quiver descriptions for $\mathscr{D}$-modules \cite{MR781776}.

Finally, switching to a homotopy-theoretic perspective, MacPherson \cite{macpherson-notes} observes that constructible sheaves, a.k.a\ perverse sheaves for the zero perversity $p(S)=0$, can be described as representations of the exit path category, directly generalising  the usual monodromy description of local systems to stratified spaces. Treumann \cite{MR2575092} further generalises by showing that perverse sheaves on a topologically stratified space can be described as representations of the exit path $2$-category. We are not aware of any explicit descriptions of perverse sheaves from this viewpoint.
 
In summary, techniques from algebra, algebraic geometry, Morse-theory, combinatorial topology and homotopy theory have all been employed to obtain alternative descriptions of categories of perverse sheaves. As mentioned before, our approach is closest to the algebraic one of Be\u{\i}linson, Ginzburg and Soergel \cite{BGS}, but applied to more general spaces and perversities. The price we pay for this generality is that there is no longer any Koszul or highest weight category structure.

In brief, \S\ref{background} sets the context and introduces notation. Section \ref{local systems} recalls the single stratum case in which perverse sheaves are local systems on a manifold and therefore have an algebraic description as modules over the fundamental group algebra. The main construction is in \S\ref{proj perv}, where we explain when there are enough projective perverse sheaves. In \S\ref{fdas} we reformulate the results of \S\ref{proj perv} in terms of  finite-dimensional algebras and also comment on computational approaches.

\section{Background}
\label{background}

\subsection{Topologically stratified spaces}

Throughout  $X$ will be a topologically stratified space in the sense of \cite{gm2}. Briefly, a $0$-dimensional topologically stratified space is a discrete union of points; a strictly positive dimensional $X$ is a paracompact Hausdorff topological space with a finite filtration
\[
\emptyset = X_{-1} \subset X_0 \subset \cdots \subset X_d = X
\]
by closed subsets such that $X_i - X_{i-1}$ is a (possibly empty) $i$-dimensional topological manifold, and such that each $x\in X_i - X_{i-1}$ has an open neighbourhood filtration-preserving homeomorphic to $\R^i \times C(L)$ for some topologically stratified space $L$. Here $C(L) = L \times [0,1)/L\times \{0\}$ is the open cone on $L$ with the induced filtration by the vertex and the subsets $L_i \times [0,1)/L_i\times \{0\}$. The stratified space $L$ is known as {\em a link} of $X_i- X_{i-1}$ at $x$. The links are not part of the data and need not even be well-defined up to homeomorphism \cite{MR319207}. The strata of $X$ are the connected components of the $X_i - X_{i-1}$. They are partially-ordered by the relation $S\leq T \iff S\subset \overline{T}$.

\subsection{The constructible derived category}

Fix a field $k$; we do not assume it has characteristic $0$, nor that it is algebraically closed. The constructible derived category $\constr{c}{X}$ is the full subcategory of the bounded derived category of sheaves of $k$-vector spaces on $X$ on those complexes whose cohomology sheaves are locally-constant on each stratum of $X$. 

Let $\jmath \colon U \hookrightarrow X$ be the inclusion of an open union of strata, and  $\imath \colon Z=X-U \hookrightarrow X$ the complementary closed inclusion. There are triangulated functors 
\[
\begin{tikzcd}
\constr{c}{Z} \ar{rr}{\imath_!=\imath_*} &&  
\constr{c}{X} \ar{rr}{\jmath^!=\jmath^*} \ar[bend right]{ll}[swap]{\imath^*} \ar[bend left]{ll}[swap]{\imath^!}
&& \constr{c}{U} \ar[bend right]{ll}[swap]{\jmath_!} \ar[bend left]{ll}[swap]{\jmath_*}
\end{tikzcd}
\]
where $\imath_!=\imath_*$ is extension by zero from a closed subset, $\jmath^!=\jmath^*$ is restriction to an open subset, $\imath^*$ and $\jmath_!$ are their respective left adjoints, and $\imath^!$ and $\jmath_*$ their respective right adjoints. (The functor $\jmath_*$ is the right derived functor of the usual sheaf theory pushforward, but since we make no use of the latter we do not use the notation $R\jmath_*$.) These functors satisfy various well-known identities, in particular $\jmath^*\imath_*=0$ and the obvious consequences for the adjoints, $\jmath^*\jmath_! = \id = \jmath^*\jmath_*$ and $\imath^!\imath_*= \id = \imath^*\imath_*$. There are two natural exact triangles $\imath_!\imath^!\sh{E} \to \sh{E} \to \jmath_*\jmath^*\sh{E} \to \imath_!\imath^!\sh{E}[1]$ and $\jmath_!\jmath^!\sh{E} \to \sh{E} \to \imath_*\imath^*\sh{E} \to \jmath_!\jmath^!\sh{E}[1]$.
The Verdier dual 
\[
\dual(-) = \Mor{\constr{c}{X}}{-}{\pi^!k_\pt} \colon \constr{c}{X}^\text{op} \to \constr{c}{X}
\]
 is a triangulated equivalence where $\pi \colon X \to \pt$ is the map to a point and $k_\pt$ the constant sheaf with stalk $k$, in degree $0$. It commutes with $\imath_*$ and $\jmath^*$ and there are natural isomorphisms $\dual\circ \imath^* \cong \imath^!\circ \dual$ and $\dual\circ\jmath_! \cong \jmath^*\circ \dual$. The above natural exact triangles are Verdier dual to one another.

\subsection{Perverse sheaves}
A \defn{perversity} on $X$ is function $p \colon \cat{S} \to \Z$ where $\cat{S}$ is the set of strata of $X$. We do not impose any further conditions, although for many applications it is useful to do so. The category $\perv{p}{X}$ of $p$-perverse sheaves is the heart of a bounded $t$-structure on $\constr{c}{X}$ obtained by `glueing' the categories $\loc{S}[-p(S)]$ of shifted local systems on the strata $S$ of $X$ --- see \cite[\S 2]{bbd} for details. It is a $k$-linear, full abelian subcategory of $\constr{c}{X}$. When $X$ has finitely many strata it is a finite length (noetherian and artinian) category.

There is a functor $\pfun{p}{H}^0 \colon \constr{c}{X} \to \perv{p}{X}$ left inverse to the inclusion which is \defn{cohomological}, \ie takes exact triangles to long exact sequences of perverse sheaves. More concretely, perverse sheaves are characterised by the vanishing conditions
\begin{equation}
\label{perverse vanishing conditions}
\sh{E} \in \perv{p}{X} \iff 
\begin{cases}
\sh{H}^d (\imath_S^*\sh{E}) = 0 & d> p(S)\\
\sh{H}^d( \imath_S^!\sh{E})=0 & d<p(S)
\end{cases}
\end{equation}
 for all strata $\imath_S \colon S \hookrightarrow X$, where $\sh{H}^d(\sh{E})$ denotes the cohomology sheaf of the complex of sheaves $\sh{E}$. 

Verdier duality on $\constr{c}{X}$ restricts to an exact equivalence
\[
\dual \colon \perv{p}{X}^\text{op} \to \perv{p^*}{X}
\]
where $p^*(S) = -\dim_\R(S) -p(S)$ is the \defn{dual perversity}. This is a generalisation of the fact that Verdier duality preserves local systems on a manifold $M$ up to a shift; specifically $\dual (\sh{L}) = \sh{L}^\vee[\dim_\R(M)]$ where  $\sh{L}^\vee$ is the dual local system. 

Extension by zero from a closed union of strata and restriction to an open union are t-exact functors. It follows that  $\imath^*$ and $\jmath_!$ are right $t$-exact for the perverse $t$-structure, and that $\imath^!$ and $\jmath_*$ are left $t$-exact. As above, let $\jmath \colon U \hookrightarrow X$ and  $\imath \colon Z=X-U \hookrightarrow X$ be complementary open and closed inclusions. Writing $\pfun{p}{\imath^*} = \pfun{p}{H}^0\circ\imath^*$ and so on there is a diagram of functors
\[
\begin{tikzcd}
\perv{}{Z} \ar{rr}{\imath_!=\imath_*} && 
\perv{}{X} \ar{rr}{\jmath^!=\jmath^*} \ar[bend right]{ll}[swap]{\pfun{p}{\imath^*}} \ar[bend left]{ll}[swap]{\pfun{p}{\imath^!}}
&& \perv{}{U} \ar[bend right]{ll}[swap]{\pfun{p}{\jmath_!}} \ar[bend left]{ll}[swap]{\pfun{p}{\jmath_*}}
\end{tikzcd}
\]
in which $\imath_!=\imath_*$ and $\jmath^!=\jmath^*$ are exact, $\pfun{p}{\imath^*}$ and $\pfun{p}{\jmath_!}$ are their respective left adjoints, and $\pfun{p}{\imath^!}$ and $\pfun{p}{\jmath_*}$ their respective right adjoints. 

\subsection{Simple perverse sheaves}

If $\jmath \colon U \to X$ is the inclusion of an open union of strata, the \defn{intermediate extension} functor $\pfun{p}{\jmath_{!*}} \colon \perv{}{U} \to \perv{}{X}$ is the image of the natural morphism $\pfun{p}{\jmath_!} \to \pfun{p}{\jmath_*}$. It is fully-faithful, preserves both monomorphisms and epimorphisms, but in general is neither left nor right exact. The intermediate extension is the unique extension with no non-zero subobjects or quotients supported on $Z$; equivalently, it is the unique extension satisfying the vanishing conditions (\ref{perverse vanishing conditions}) for the corresponding non-strict inequalities. The simple perverse sheaves are those of the form ${\imath_S}_*\pfun{p}{{\jmath_S}_{!*}}\sh{L}[-p(S)]$ where $\sh{L}$ is an irreducible local system on a stratum $S$ and $\jmath_S \colon S \hookrightarrow \overline{S}$ and $\imath_S \colon \overline{S} \hookrightarrow X$ are the inclusions. These are known as \defn{intersection cohomology complexes} because their cohomology groups are, up to a shift by $p(S)$ in the indexing, the perversity $p$ intersection cohomology groups of the closure $\overline{S}$ with coefficients in $\sh{L}$. We therefore use the notation $\ic{p}{\sh{L}} =  {\imath_S}_*\pfun{p}{{\jmath_S}_{!*}}\sh{L}[-p(S)]$.

\section{Local systems}
\label{local systems}

A stratified space $X$ with a single stratum is a topological manifold. In this case, for any perversity $p$, the perverse sheaves $\perv{p}{X}$ are equivalent to the category $\loc{X}$ of finite-dimensional local systems on  $X$ with coefficients in $k$, \ie the category of locally-constant sheaves of finite-dimensional $k$-vector spaces on $X$. 

Taking  monodromy establishes an equivalence $\loc{X} \simeq \mod{k[\pi_1X]}$ with the category of finite-dimensional left modules over the group algebra of the fundamental group. In particular the question of whether there are enough projective local systems on $X$ depends only on the fundamental group. Clearly, if $\pi_1X$ is finite then $k[\pi_1X]$ is a finite-dimensional $k$-algebra and $\mod{k[\pi_1X]}$ has enough projectives. Generically one has a stronger result --- if $\pi_1X$ is finite and the characteristic of $k$ does not divide its order then Maschke's Theorem implies that $\mod{k[\pi_1X]}$ is a semi-simple category so that all modules are projective. However, if $\pi_1X\cong \Z$ and $k=\C$, for example, then there are no non-zero projective local systems. In this case indecomposable local systems are classified by their Jordan normal forms, none of which is projective. In general the question is quite subtle (and we do not attempt to answer it). To see why recall that there are finitely-presented infinite groups with no non-trivial finite-dimensional representations (over any field $k$). For example a representation of a finitely-presented infinite {\em simple} group $G$ is either trivial or faithful, but the latter is impossible since $G$ cannot embed as a subgroup of $GL_n(k)$. (This is because a  finitely-generated  linear group is residually finite by Mal'cev's Theorem, \ie the intersection of all its finite index normal subgroups is trivial, and so it cannot be simple.) Since any finitely-presented group arises as the fundamental group of a smooth compact $4$-manifold $X$, we can find such an $X$ with $\pi_1X \cong G$ and therefore with  $\loc{X}$ equivalent to the category of finite-dimensional $k$-vector spaces.

After this detour into the intricacies of representation theory the main result of this paper should come as a relief! Roughly, it says that generalising to stratified spaces one does not meet any further subtleties; the existence of projective perverse sheaves depends only on the fundamental groups of the strata.

\section{Projective perverse sheaves}
\label{proj perv}

In this section we assume that the topologically stratified space $X$ has finitely many strata.  Under this assumption, $\perv{p}{X}$ is a Hom-finite length abelian category for any perversity $p$, \ie the Hom-spaces between objects are finite-dimensional vector spaces over $k$, and each object has a finite composition series with simple quotients. Either of these properties implies that $\perv{p}{X}$ is a Krull-Remak-Schidt category \cite[\S 5]{MR3431480} --- each perverse sheaf is a finite direct sum of indecomposable perverse sheaves, and a perverse sheaf is indecomposable precisely when its endomorphism algebra is local. The indecomposable summands of a perverse sheaf, counted with multiplicity, are unique up to isomorphism and reordering. In particular, $\perv{p}{X}$ has enough projectives if and only if each simple object has a projective cover, and the projective covers of the simple objects are precisely the indecomposable projectives \cite[Lemma 3.6]{MR3431480}. We will use the following well-known characterisation of indecomposable projective objects.
\begin{lemma}
\label{proj cover criteria}
Fix a simple perverse sheaf $\,\ic{p}{\sh{L}}$. If $\sh{P}$ satisfies $\ext{1}{\sh{P}}{\ic{p}{\sh{M}}}=0$ for all simple perverse sheaves $\ic{p}{\sh{M}}$  and
\begin{equation}
\label{hom conditions}
\mor{}{\sh{P}}{\ic{p}{\sh{M}}} \cong 
\begin{cases}
k & \text{if}\ \sh{M}\cong \sh{L}\\
0 & \text{otherwise.}
\end{cases}
\end{equation}
then $\sh{P}$ is a projective cover of $\ic{p}{\sh{L}}$.
\end{lemma}
\begin{proof}
The first condition implies that $\sh{P}$ is projective. The second condition implies it has $\ic{p}{\sh{L}}$ as a quotient and that $\sh{P}$ is indecomposable; if it decomposes only one summand can have $\ic{p}{\sh{L}}$ as a quotient, and the other summand has no non-zero quotients at all, hence vanishes. Therefore $\sh{P}$ is an indecomposable projective with $\ic{p}{\sh{L}}$ as a quotient, and so is a projective cover of $\ic{p}{\sh{L}}$.
\end{proof}

\subsection{Restricting and extending projective perverse sheaves}
Let $\jmath \colon U \hookrightarrow X$ be the inclusion of an open union of strata and $\imath \colon Z \hookrightarrow  X$ the complementary closed union. 
\begin{lemma}
\label{proj extensions and restrictions}
The functors $\pfun{p}{\jmath_!}$ and $\pfun{p}{\imath^*}$ preserve projective perverse sheaves. 
 \end{lemma}
\begin{proof}
This follows because the left adjoint of an exact functor preserves projective objects.
\end{proof}
\begin{lemma}
\label{proj restriction}
Suppose $\sh{P}\in \perv{}{X}$ is projective and that $\pfun{p}{\imath^*}\sh{P}=0$. Then $\jmath^*\sh{P}\in \perv{}{U}$ is projective. 
\end{lemma}
\begin{proof}
We show that $\mor{}{\jmath^*\sh{P}}{-}\cong \mor{}{\sh{P}}{\pfun{p}{\jmath_*}(-)}$ is exact. Given short exact $0\to \sh{E}\to\sh{F}\to\sh{G}\to 0$ in $\perv{}{U}$ the cokernel of $\pfun{p}{\jmath_*}\sh{F}\to\pfun{p}{\jmath_*}\sh{G}$ is supported on $Z$ so that there is an exact sequence $0\to \pfun{p}{\jmath_*}\sh{E}\to\pfun{p}{\jmath_*}\sh{F}\to\pfun{p}{\jmath_*}\sh{G}\to \imath_*\sh{C}\to 0$  for some $\sh{C}\in \perv{}{Z}$. Applying the exact functor $\mor{}{\sh{P}}{-}$ we obtain a short exact sequence
\[
0 \to 
\mor{}{\sh{P}}{\pfun{p}{\jmath_*}\sh{E}}\to
\mor{}{\sh{P}}{\pfun{p}{\jmath_*}\sh{F}}\to
\mor{}{\sh{P}}{\pfun{p}{\jmath_*}\sh{G}}\to0
\]
since $\mor{}{\sh{P}}{\imath_*\sh{C}} \cong \mor{}{\pfun{p}{\imath^*}\sh{P}}{\sh{C}} =0$.
\end{proof}

A simple perverse sheaf on $X$ is either the intermediate extension $\pfun{p}{\jmath_{!*}}\sh{E}$ of a simple perverse sheaf on $U$ or the extension by zero $\imath_*\sh{E}$ of a simple perverse sheaf on $Z$. This establishes a correspondence between the set of isomorphism classes of simple perverse sheaves on $X$ and the union of the sets of isomorphism classes of perverse sheaves on $Z$ and on $U$. 

\begin{lemma}
\label{preserves proj covers}
\begin{enumerate}
\item Suppose $\sh{E}\in \perv{p}{U}$ is simple and $\sh{P}$ is a projective cover of $\sh{E}$. Then $\pfun{p}{\jmath_!}\sh{P}$ is a projective cover of the intermediate extension $\pfun{p}{\jmath_{!*}}\sh{E}$.
\item Suppose $\sh{E}\in \perv{p}{U}$ is simple and $\sh{P}$ is a projective cover of its intermediate extension $\pfun{p}{\jmath_{!*}}\sh{E}$. Then $\jmath^*\sh{P}$ is a projective cover of $\sh{E}$.
\item Suppose $\sh{E} \in \perv{p}{Z}$ is simple and $\sh{P}$ is a projective cover of its extension by zero $\imath_*\sh{E}$. Then $\pfun{p}{\imath^*}\sh{P}$ is a projective cover of $\sh{E}$.
\end{enumerate}
\end{lemma}
\begin{proof}
In each case we apply Lemma \ref{proj cover criteria}. 
\begin{enumerate}
\item  By Lemma \ref{proj extensions and restrictions} the extension $\pfun{p}{\jmath_!}\sh{P}$ is projective, so the Ext condition in Lemma \ref{proj cover criteria} holds. The Hom conditions in (\ref{hom conditions})  follow from the adjunction between $\pfun{p}{\jmath_!}$ and $\jmath^*$.
\item Since $\mor{}{\sh{P}}{\imath_*\sh{F}}=0$ for all simple $\sh{F} \in \perv{p}{Z}$ we have $\pfun{p}{\imath^*}\sh{P}=0$. Hence $\jmath^*\sh{P}$ is projective by Lemma \ref{proj restriction}, so the Ext condition in Lemma \ref{proj cover criteria} holds. The Hom conditions in (\ref{hom conditions}) follow from the adjunction between $\jmath^*$ and $\pfun{p}{\jmath_{*}}$ and the fact that for simple $\sh{F}\in \perv{p}{U}$ there is a short exact sequence $0\to \pfun{p}{\jmath_{!*}}\sh{F} \to \pfun{p}{\jmath_{*}}\sh{F} \to \imath_*\sh{F}' \to 0$ with $\sh{F}'\in \perv{p}{Z}$.
\item By Lemma \ref{proj extensions and restrictions} the restriction $\pfun{p}{\imath^*}\sh{P}$ is projective, so the Ext condition in Lemma \ref{proj cover criteria} holds. The Hom conditions in (\ref{hom conditions})  follow from the adjunction between $\pfun{p}{\imath^*}$ and $\imath_*$. \qedhere
\end{enumerate}
\end{proof}
In the next section we establish a sufficient criterion for $\perv{p}{X}$ to have enough projectives. Here we note  a  necessary condition implied by  the previous lemma.
\begin{corollary}
\label{necessary}
Suppose $X$ is a topologically stratified space with finitely many strata and $p$ a perversity. If $\perv{p}{X}$ has enough projectives then so does $\loc{S}$ for each stratum $S$. 
\end{corollary}
\begin{proof}
Let $\sh{E}\in \loc{S}$ be irreducible and suppose that $\sh{P}$ is a projective cover of $\ic{p}{\sh{E}}$ in $\perv{p}{X}$. Lemma  \ref{preserves proj covers} implies that ${\jmath_S}^*\pfun{p}{\imath_S^*}\sh{P}$ (shifted by $p(S)$) is a projective cover of $\sh{E}$ in $\loc{S}$ where $\jmath_S \colon S \hookrightarrow \overline{S}$ and $\imath_S \colon \overline{S} \hookrightarrow X$ are the inclusions.
\end{proof}

\subsection{Existence of projectives}

The next result establishes a sufficient condition for $\perv{p}{X}$ to have enough projectives. The delicate part is the construction of a projective cover of a simple perverse sheaf supported on a closed stratum. This construction generalises that of \cite[Thm 3.2.1]{BGS}. It mimics the construction of a projective cover of simple a module for the path algebra of a quiver.
\begin{theorem}
\label{sufficient}
Let $X$ be a topologically stratified space and $p$ a perversity. Suppose $X$ has finitely many strata and that $\loc{S}$ has finitely many (isomorphism classes of) simple objects and enough projectives for each stratum $S$. Then $\perv{p}{X}$ also has enough projectives.
\end{theorem}
\begin{proof}
It is enough to construct a projective cover for each simple perverse sheaf. We do so by induction over the number of strata. When $X$ has a single stratum there is nothing to prove. Suppose $S$ is a closed stratum and let $\imath \colon S \hookrightarrow X$ and $\jmath \colon X-S \hookrightarrow X$ be the inclusions. Since $X-S$ has strictly fewer strata we may assume that each simple $\jmath^*\ic{p}{\sh{M}}$  has a projective cover $\sh{A}_\sh{M}$ in $\perv{p}{X-S}$. Then $\sh{P}_\sh{M}=\pfun{p}{\jmath_!}\sh{A}_\sh{M}$ is a projective cover of $\ic{p}{\sh{M}}$  by Lemma \ref{preserves proj covers}. 

It remains to construct a projective cover for $\ic{p}{\sh{L}}$ where $\sh{L}$ is an irreducible local system on $S$. By assumption $\sh{L}$ has a projective cover in $\loc{S}$. Let $\sh{B}_\sh{L}$ denote the perverse sheaf obtained by shifting this cover by $[-p(S)]$ and extending by zero. Let $I$ be the union of the sets of (isomorphism classes of) irreducible local systems on the strata of $X-S$. Since $I$ is finite 
\[
\sh{P} =  \bigoplus_{\sh{M} \in I} \proj{}{\sh{M}} \otimes \ext{1}{\sh{B}_\sh{L}}{\ic{p}{\sh{M}}}^\vee
\]
is a well-defined and projective perverse sheaf. Let $\pi \colon \sh{P}\epic \sh{Q}$ be such that $\sh{Q}$ has maximal length amongst quotients for which there exists $\epsilon \in \ext{1}{ \sh{B}_\sh{L}}{ \sh{Q}}$ inducing isomorphisms
\begin{equation}
\label{q surj}
\mor{}{\sh{Q}}{\ic{p}{\sh{N}}} \cong \ext{1}{\sh{B}_\sh{L}}{ \ic{p}{\sh{N}}} \colon \varphi \mapsto \varphi \circ \epsilon
\end{equation}
for each $\sh{N}\in I$. Such a $\sh{Q}$ exists because $\sh{P}$ has finite length and the quotient
\[
 \bigoplus_{\sh{M}\in I} \ic{p}{\sh{M}} \otimes \ext{1}{\sh{B}_\sh{L}}{\ic{p}{\sh{M}}}^\vee
\]
has the required property --- a suitable choice of $\epsilon$  in this case is the sum of the units in
\[
 \bigoplus_{\sh{M}\in I} \ext{1}{\sh{B}_\sh{L}}{\ic{p}{\sh{M}}} \otimes \ext{1}{\sh{B}_\sh{L}}{\ic{p}{\sh{M}}}^\vee.
\]
Let $\proj{}{\sh{L}} \in \perv{}{X}$ be defined (up to isomorphism) by the triangle
\begin{equation}
\label{proj triangle}
\sh{Q}\to \proj{}{\sh{L}} \to  \sh{B}_\sh{L} \stackrel{\epsilon}{\longrightarrow}  \sh{Q}[1].
\end{equation}
We prove that $\proj{}{\sh{L}}$ is the projective cover of $\ic{p}{\sh{L}}$ using Lemma \ref{proj cover criteria}. Apply $\mor{}{-}{\ic{p}{\sh{N}}}$ to (\ref{proj triangle}) to obtain a long exact sequence. The property (\ref{q surj}) implies this splits into an isomorphism $\mor{}{ \proj{}{\sh{L}} }{ \ic{p}{\sh{N}} } \cong \mor{}{  \sh{B}_\sh{L} }{\ic{p}{\sh{N}}} \cong \mor{}{  \sh{B}_\sh{L}}{\pfun{p}{\imath^*} \ic{p}{\sh{N}} }$ and an exact sequence
\begin{equation}
\label{ext seq}
0
 \to \ext{1}{\sh{P}_\sh{L}}{\ic{p}{\sh{N}}} 
 \to \ext{1}{\sh{Q}}{\ic{p}{\sh{N}}} 
 \to \ext{2}{ \sh{B}_\sh{L}}{ \ic{p}{\sh{N}}}\to  \cdots,  
 \end{equation}
 where the Ext-groups are computed in $\constr{c}{X}$ rather than in $\perv{p}{X}$. Since
 \[
\pfun{p}{\imath^*} \ic{p}{\sh{N}} \cong
\begin{cases}
\sh{N}[-p(S)] & \text{if}\ \sh{N} \in \loc{S} \\
0 & \text{otherwise},
\end{cases}
\]
 the only simple quotient of $\sh{P}_\sh{L}$ is $\ic{p}{L}$, and this occurs with multiplicity one. Hence $\sh{P}_\sh{L}$ is indecomposable. To show $\sh{P}_\sh{L}$ is projective it suffices to prove that the third map in (\ref{ext seq}) is injective. Suppose $0\neq \varphi \in \ext{1}{\sh{Q}}{\ic{p}{\sh{N}}}$ is in the kernel, \ie $\varphi \circ \epsilon[-1]=0$. Then we have a commutative diagram 
\[
\begin{tikzcd}
{}&& \sh{P} \ar[->>]{d}{\pi} \ar[dashed,swap]{dl}{\pi'} \ar{dr}{0}&\\
\ic{p}{\sh{N}} \ar{r} & \sh{Q}' \ar{r} &\sh{Q} \ar{r}{\varphi} & \ic{p}{\sh{N}}[1]\\
&&   \sh{B}_\sh{L}[-1] \ar{u}[swap]{\epsilon[-1]}\ar[dashed]{ul}{\epsilon'[-1]} \ar{ur}[swap]{0}&
\end{tikzcd}
\]
in $\constr{c}{X}$ whose middle row is the triangle induced from $\varphi$. The composite $\varphi\circ \pi=0$ because $\sh{P}$ is projective. Therefore there are factorisations via $\pi'$ and $\epsilon'[-1]$ as indicated. By construction $\sh{Q}'$ is a perverse sheaf of greater length than $\sh{Q}$. Applying $\mor{}{-}{\ic{p}{\sh{M}}}$ to the triangle induced by $\varphi$ yields a long exact sequence
\[
0 
\to \mor{}{\sh{Q}}{\ic{p}{\sh{M}}} 
\to \mor{}{\sh{Q}'}{\ic{p}{\sh{M}}}   
\to \mor{}{\ic{p}{\sh{N}}}{\ic{p}{\sh{M}}} 
\to \ext{1}{\sh{Q}}{\ic{p}{\sh{M}}} 
\to \cdots.
\]
Since the third term vanishes for $\sh{M}\neq \sh{N}$, and injects into the fourth when $\sh{M}= \sh{N}$ because $\varphi\neq 0$, we conclude that $\mor{}{\sh{Q}}{\ic{p}{\sh{M}}} \cong \mor{}{\sh{Q}'}{\ic{p}{\sh{M}}}$. Therefore composition with $\epsilon'$ induces an isomorphism
\[
\mor{}{\sh{Q}'}{\ic{p}{\sh{M}}} \cong  \ext{1}{\sh{B}_\sh{L} }{\ic{p}{\sh{M}}} 
\]
for any $\sh{M}\in I$. Thus $\pi'$ cannot be an epimorphism of perverse sheaves, for otherwise $\sh{Q}$ would not be the maximal length quotient of $\sh{P}$ satisfying (\ref{q surj}). Since $\ic{p}{\sh{N}}$ is simple  $\im\, \pi' \cong \sh{Q}$ and $\sh{Q}'\to \sh{Q}$ splits in $\perv{p}{X}$ contradicting the fact that $\varphi\neq 0$. We conclude that the third map in (\ref{ext seq}) is injective, and this completes the proof.
\end{proof}
\begin{remark}
\label{identifying Q}
It is not clear {\em a priori} that the quotient $\sh{Q}$ appearing in the proof is unique (up to isomorphism), however {\em a posteriori} we see that it is. The short exact sequence $0 \to \sh{Q} \to \sh{P}_\sh{L} \to \sh{B}_\sh{L}\to 0$ and the fact that Lemma \ref{proj extensions and restrictions} implies $\sh{B}_\sh{L} \cong \imath_*\pfun{p}{\imath^*}\sh{P}_\sh{L}$  show that $\sh{Q} \cong \pfun{p}{\jmath_!}\jmath^*\sh{P}_\sh{L}$.
\end{remark}

What is surprising about this result is that the existence of enough projective perverse sheaves depends only upon the fundamental groups of the strata and the field $k$, and not on the perversity or on any information about how the strata are assembled to form the space. In fact, although we have formulated it in terms of perverse sheaves, it can be reformulated in the abstract setting of recollement of t--structures. 

\begin{proposition}
Suppose $\cat{D}_Z \stackrel{\imath_*}{\longrightarrow} \cat{D}_X \stackrel{\jmath^*}{\longrightarrow} \cat{D}_U$ is an exact triple of $k$-linear Hom-finite triangulated categories satisfying the axioms for recollement \cite[\S 1.4.3]{bbd}. Further suppose we have bounded t--structures on $\cat{D}_Z$ and $\cat{D}_U$ whose hearts are length categories with finitely many simple objects. Then the heart of the glued t--structure on $\cat{D}_X$ \cite[Thm 1.4.10]{bbd} has enough projectives if and only if the hearts of the t--structures on $\cat{D}_Z$ and $\cat{D}_U$ each have enough projectives.
\end{proposition}

\section{Finite-dimensional algebras}
\label{fdas}

When there are enough projective perverse sheaves and finitely many simple ones perverse sheaves can be described as modules over a finite-dimensional algebra, and therefore also as representations of a quiver with relations. The direct sum of projective covers of the simple perverse sheaves is a projective generator of $\perv{p}{X}$ and tilting theory provides an equivalence between $\perv{p}{X}$ and finite-dimensional modules over its endomorphism ring. More precisely we apply the following result.
\begin{theorem}\cite[Chapter II, Exercise after Theorem 1.3]{MR0249491}
\label{proj gen thm}
Let $\cat{C}$ be a Hom-finite and length $k$-linear abelian category. Then $\cat{C}$ has a projective generator if and only if there is an exact equivalence $\cat{C} \simeq \mod{A}$ where $\mod{A}$ is the category of finite-dimensional (left) modules over a finite-dimensional $k$-algebra $A$.
\end{theorem}

\begin{corollary}
\label{main result}
Let $X$ be a topologically stratified space and $p$ a perversity on $X$. There is an exact equivalence $\perv{}{X} \simeq \mod{A}$ where $A$ is a finite-dimensional $k$-algebra if and only if $X$ has finitely many strata and for each stratum $S$ there is a finite-dimensional $k$-algebra $A_S$ with an exact equivalence $\loc{S} \simeq \mod{A_S}$. 
\end{corollary}
\begin{proof}
Recall that there are finitely many simple modules over any finite-dimensional algebra, and note that $\perv{p}{X}$ has finitely many simple objects if and only if $X$ has finitely many strata each with only finitely many irreducible local systems. The result follows by combining Theorems \ref{sufficient} and \ref{proj gen thm} with Corollary \ref{necessary}.
\end{proof}

We emphasise that the perversity $p$ and the links of the stratification of $X$ play no role. Of course, these do enter into the determination of an algebra $A$ whose module category is the perverse sheaves. An immediate consequence is that $\perv{p}{X}$ has a projective generator  if and only if $\perv{p^*}{X}$, where $p^*$ is the dual perversity, has one. Hence, by duality, $\perv{p}{X}$ has  an injective cogenerator if and only if it has a projective generator.

The immediate corollary identifies a large class of examples.
\begin{corollary}
Suppose $X$ is a topologically stratified space with finitely many strata, each with finite fundamental group, and $p$ any perversity. Then the category $\perv{p}{X}$ of perverse sheaves with coefficients in a field $k$ is equivalent to the category of finite-dimensional (left) modules over a finite-dimensional $k$-algebra.
\end{corollary}

\subsection{Remarks on computations}

How can one find, when it exists, a finite-dimensional $k$-algebra $A$ such that $\perv{p}{X} \simeq \mod{A}$? As mentioned in the introduction there are several known approaches, but these  involve extra geometric assumptions on $X$ and restrictions on the perversity $p$. The constructions and results above open the possibility of more algebraic approaches. We outline three of these. A second paper will give detailed examples.

The first approach is the most direct. The proof of Theorem \ref{sufficient} explains how to inductively construct  a projective cover $\sh{P}_\sh{L}$ of a simple perverse sheaf $\ic{p}{\sh{L}}$. If one can do this then the sum $\bigoplus_\sh{L} \sh{P}_\sh{L}$ is a projective generator, and the algebra we seek is its endomorphism ring. Unfortunately, it is not easy to implement this construction of $\sh{P}_\sh{L}$ as an effective algorithm. The principal obstruction is that one has to find a maximal length quotient $\sh{Q}$ of
\[
\sh{P}=\bigoplus_{\sh{M} \in I} \proj{}{\sh{M}} \otimes \ext{1}{\sh{B}_\sh{L}}{\ic{p}{\sh{M}}}^\vee
\]
satisfying the property (\ref{q surj}). This quotient exists, and is unique by Remark \ref{identifying Q}, but we do not have a better construction than searching through all the quotients. For the top `classical' perversity $p(S)=-\dim_\R(S)$ the simple perverse sheaves have the form ${\jmath_S}_*\sh{L}[-\dim_\R(S)]$, where $\jmath_S \colon S \hookrightarrow \overline{S}$. This implies that $\sh{P}=0$ whence also $\sh{Q}=0$. So in this case the construction degenerates and the projective cover $\sh{P}_\sh{L} = \pfun{p}{{\jmath_S}_!}\sh{L}[-\dim_\R(S)]$. However, in general there can be multiple quotients satisfying  (\ref{q surj}). For example, let $X=\C\PP^1$ stratified by $\C$ and a point $\infty$, let  $p(S) = -\dim_\C(S)$ be the middle perversity, and $\sh{L} = k_\infty$ be the skyscraper on the point stratum. Then $\sh{P}=\jmath_!k_\C[1]$ and both itself and its quotient $\pfun{p}{\jmath_{!*}}k_\C[1]$ satisfy (\ref{q surj}). Choosing the maximal length quotient, as the construction dictates, gives $\sh{Q}=\sh{P}=\jmath_!k_\C[1]$ and verifies (as is well-known) that Be\u{\i}linson's maximal extension is the projective cover of $\imath_*k_\infty$.

The second approach is to try to obtain a quiver description. When $k$ is algebraically closed, the category of finite-dimensional modules over a finite-dimensional algebra is equivalent to the category of finite-dimensional representations of a finite quiver with admissible relations \cite[Chapter II, Theorem 3.7]{assem_skowronski_simson_2006}. Therefore, when $k$ is algebraically closed and $\perv{p}{X}$ has a projective generator, perverse sheaves have a quiver description. The quiver is the Ext-quiver --- it has one vertex for each isomorphism class of simple objects (of which there are finitely many, labelled by irreducible local systems on the strata) and $\dim_k \ext{1}{\ic{p}{\sh{L}}}{\ic{p}{\sh{M}}}$ arrows from the vertex labelled by $\sh{L}$ to that labelled by $\sh{M}$. These groups can be computed inductively in terms of intersection cohomology groups of links, or by using a spectral sequence \cite[\S 3.4]{BGS}. The relations are determined by the canonical $A_\infty$-structure on the algebra $\Ext{\perv{p}{X}}{*}{\sh{S}}{\sh{S}}$ where $\sh{S}$ is the direct sum of the simple objects \cite[\S2.8.4]{MR2264803}. This raises two difficulties. Firstly if the perverse heart is not faithful then the underlying algebra is not the same as $\Ext{\constr{c}{X}}{*}{\sh{S}}{\sh{S}}$ in higher degrees. Whilst the latter can be computed within $\constr{c}{X}$, and thereby directly related to the topology of $X$, the former is much less accessible. Secondly, the $A_\infty$-structure is itself hard to construct. One can obtain the quadratic part of the relations from (the dual of) the composition 
\[
\ext{1}{\sh{S}}{\sh{S}} \otimes \ext{1}{\sh{S}}{\sh{S}} \to \Ext{\perv{p}{X}}{2}{\sh{S}}{\sh{S}} \hookrightarrow \Ext{\constr{c}{X}}{2}{\sh{S}}{\sh{S}}
\]
using the fact that the second map is injective \cite[Lemma 2.3]{BGS}, \ie from composition of morphisms in $\constr{c}{X}$.
 In several very interesting examples the $A_\infty$-structure is formal and all relations are quadratic so this suffices --- see \eg  \cite{BGS} and \cite{MR2264803}. However in general  the $A_\infty$-structure is non-formal and there are also higher relations which are difficult  to compute.

The third approach is via silting theory. Assume that $\perv{f}{X}$ is a faithful heart for some perversity $f$. This is the case for instance if $X$ is
\begin{enumerate}
\item  a complex projective variety stratified by affine subvareities $S$ with $H^{>0}(S;k)=0$ and $f(S) = -\dim_\C(S)$ is the middle perversity \cite[\S1.5]{MR2119139}; or
\item  a compact space stratified by a simplicial triangulation and $f$ is a `classical' perversity, \ie $f(S)=f\left(\dim_\R(S)\right)$ satisfies $f(0)=0$ and $m-n \leq f(n)-f(m) \leq 0$ \cite[Thm 4.2]{MR1453053}.
\end{enumerate}
Length hearts in $\constr{c}{X} \simeq \cat{D}^b\!\left( \perv{f}{X} \right)$ correspond to silting objects in the bounded homotopy category $\cat{K}^b\!\left(\projperv{f}{X}\right)$ of projective perverse sheaves \cite{MR3178243}. In our setting, a faithful heart has global dimension bounded by $\dim_\R(X)$, in particular it is finite, so that the canonical functor
\[
\cat{K}^b\!\left(\projperv{f}{X}\right) \to \cat{D}^b\!\left( \perv{f}{X} \right)
\]
is an equivalence. Thus there is a correspondence between length hearts and silting objects in $\constr{c}{X}$. Moreover, this correspondence is compatible with silting mutation and simple Happel--Reiten--Smal\o\ tilting. 

Since each perverse heart $\perv{p}{X}$ is length, each corresponds to a silting object  $\sh{S}_p$. The latter can be obtained by  starting with  a basic projective generator of $\perv{f}{X}$ and performing a sequence of silting mutations corresponding to a sequence of simple tilts leading from $\perv{f}{X}$ to $\perv{p}{X}$. Such a sequence always exists --- if perversities $p$ and $q$ differ by $1$ on a single stratum then the corresponding hearts are related by a Happel--Reiten-Smal\o\ tilt, which can be decomposed into a finite sequence of simple tilts. The perverse cohomology $\pfun{p}{H}^0(\sh{S}_p)$ is a projective generator of $\perv{p}{X}$, and $\perv{p}{X}$ is faithful precisely when $\sh{S}_p$ is tilting, equivalently when $\pfun{p}{H}^0(\sh{S}_p)  \cong \sh{S}_p$. Even if $\sh{S}_p$ is not tilting, there is an algebra isomorphism $\End{\sh{S}_p}\cong \End{\pfun{p}{H}^0(\sh{S}_p)}$ so that  $\perv{p}{X} \simeq \mod{\End{\sh{S}_p}}$. In summary, this approach is productive if there is a faithful heart $\perv{f}{X}$ for which we can compute a basic projective generator.

%
%

\end{document}